\documentclass[10pt,twocolumn,twoside]{IEEEtran}
\pdfoutput=1

\usepackage[T1]{fontenc}
\usepackage[normalem]{ulem}
\usepackage[utf8]{inputenc}
\usepackage{amsmath,graphicx,ogonek}
\usepackage{amsmath}
\usepackage{amssymb}
\usepackage{mathabx}
\usepackage{amsthm}
\usepackage{capt-of}
\usepackage{centernot}
\usepackage{comment}
\usepackage{dsfont}
\usepackage{graphicx}
\usepackage{grffile}
\usepackage{hyperref}
\usepackage{lineno,hyperref}
\usepackage{longtable}
\usepackage{mathrsfs} 
\usepackage{mathtools}
\usepackage{rotating}
\usepackage{textcomp}
\usepackage{wrapfig}
\usepackage{float}
\usepackage{wasysym}

\modulolinenumbers[5]
\hypersetup{colorlinks=true}

\newcommand{\suchthat}[0]{\mid}
\newtheorem{assumption}{Assumption}
\newtheorem{example}{Example}

\newtheorem{remark}{Remark}

\newtheorem{definition}{Definition}

\newtheorem{lemma}{Lemma}
\newtheorem{theorem}{Theorem}
\newtheorem{corollary}{Corollary}
\makeatletter
\newcommand{\proofpart}[2]{%
  \par
  \addvspace{\medskipamount}%
  \noindent\emph{Part #1: #2}\par\nobreak
  \addvspace{\smallskipamount}%
  \@afterheading
}
\makeatother

\makeatletter
\renewcommand*{\descriptionlabel}[1]{\xdef\@currentlabel{#1}%
\hspace\labelsep\normalfont\bfseries #1}
\makeatother

\begin{document}
\title{Fixed Points of Cone Mapping with the Application to Neural Networks}
\author{Grzegorz Gabor$^{1}$ and Krzysztof Rykaczewski$^{\star,2}$\\
  $^{1,2}$ Faculty of Mathematics and Computer Science,\\
  Nicolaus Copernicus University,
  Chopina 12/18, 87-100 Toruń, Poland\\
}

\maketitle

\begin{abstract}
  We derive conditions for the existence of fixed points of cone mappings without assuming scalability of functions.
  Monotonicity and scalability are often inseparable in the literature in the context of searching for fixed points of interference mappings.
  In applications, such mappings are approximated by non-negative neural networks.
  It turns out, however, that the process of training non-negative networks requires imposing an artificial constraint on the weights of the model.
  However, in the case of specific non-negative data, it cannot be said that if the mapping is non-negative, it has only non-negative weights.
  Therefore, we considered the problem of the existence of fixed points for general neural networks, assuming the conditions of tangency conditions with respect to specific cones.
  This does not relax the physical assumptions, because even assuming that the input and output are to be non-negative, the weights can have (small, but) less than zero values.
  Such properties (often found in papers on the interpretability of weights of neural networks) lead to the weakening of the assumptions about the monotonicity or scalability of the mapping associated with the neural network.
  To the best of our knowledge, this paper is the first to study this phenomenon.
\end{abstract}


*Corresponding author.

\IEEEpeerreviewmaketitle

\section{Introduction}



The objective of this section is to develop part of the mathematical machinery required for the applications in this study. 

Let $(\mathbb{R}^{N}, \|\cdot\|)$ be an Euclidean space. By $\operatorname{int} K$ we mean the {\em interior} of subset $K \subset \mathbb{R}^{N}$.
For $\varepsilon > 0$ and $x \in \mathbb{R}^{N}$ denote {\em open ball} as $\mathcal{B}(x, \varepsilon) := \{a \in \mathbb{R}^{N} \suchthat \|a - x\| < \varepsilon\}$.
{\em Closed ball} is defined as closure of open ball $\mathcal{D}(0, \varepsilon) := \operatorname{cl} \mathcal{B}(0, \varepsilon)$.
Boundary of set $A$ is denoted $\partial A$.

We shall discuss cones and order relations in $\mathbb{R}^{N}$.
Let $\mathbb{R}_{+} := \{x \geqslant 0 \suchthat x \in \mathbb{R}\}$ be set of non-negative real numbers.
A nonempty closed and convex subset $K \subseteq \mathbb{R}^{N}$ is called a cone, if $\mathbb{R}_{+} K \subseteq K$ and $K \cap \left(-K\right) = \{0\}$\footnote{Sometimes it is called {\em pointed cone}.} hold.
Convex cone $K$ satisfy $K + K \subseteq K$.

Let us assume $K \neq \{0\}$ throughout.
For vectors $x, x^{\prime} \in \mathbb{R}^{N}$ we introduce the relations
\begin{equation}
  \begin{array}{lll}
    x \leqslant_K x^{\prime} & \text{if and only if} & x^{\prime} - x \in K, \\
    x <_K x^{\prime} & \text{if and only if} & x^{\prime} - x \in K \backslash\{0\}, \\
    x \ll_K x^{\prime} & \text{if and only if} & x^{\prime} - x \in \operatorname{int} K;
  \end{array}
\end{equation}
the latter one requires $K$to have nonempty interior; one speaks of a {\em solid} cone $K$.
Equipped with such a order, $(\mathbb{R}^{N}, \leqslant_K)$ is called an {\em ordered Euclidean space}.

\begin{remark}
  Relation $x <_{K} x^{\prime}$ means that $x \leqslant_{K} x^{\prime}$ and $x \neq x^{\prime}$.

  Relation $0 \leqslant_K x$ is equivalent to $x \in K$.

  In the special case where $K = \mathbb{R}_{+}^{N}$ we note that $x \leqslant_K y$ if and only if $x_{i} \leq y_{i}$, for all $1 \leq i \leq n$. 
  Moreover, in this context $x = (x_1, \ldots, x_N) \gg_K 0$ means that $x_i > 0$, $i = 1, \ldots, N$.
\end{remark}

\begin{example}
  The nonegative orthant $K := \mathbb{R}^{N}_{+} = \left\{x = \left(x_{1}, \ldots, x_{N}\right) \suchthat x_{i} \geqslant 0, i = 1, \ldots, N\right\} \subset \mathbb{R}^{N}$ denotes cone consisting of points with all coordinates nonegative, called the {\em positive cone}. 
  Using this cone we get the partial order $\leqslant_K$ where $x\leqslant_K y$ iff $x_i \leq y_i$, for all $i \in \{ 1, \ldots, N\}$. In this case, subscript $K$ is often omitted.
  If all coordinates of vector $x \in \mathbb{R}^{N}$ are positive we write $x \gg 0$ or $x \gg_{\mathbb{R}^N_{+}} 0$.
\end{example}

\begin{example}

  Let $w \in \mathcal{H}$, $w \neq 0$.
  Another important cone in $\mathbb{R}^{N}$ generating partial order is
  \begin{equation}
    K := C(w, \beta) = \big\{ v \suchthat \langle v, w\rangle \geq \beta \|v\|\, \|w\| \big\},
  \end{equation}
  where $\beta \in [0, 1]$.
  Geometrically, inequality $x\leqslant_K y$ means that the vector $y - x$ makes with vector $w$ an angle less than or equal to $\alpha := \arccos\beta \in \left[0, \frac{\pi}{2}\right]$:
  \begin{equation}
    C\big(w, \cos(\alpha)\big) = \{z \in \mathbb{R}^{N} \suchthat \varangle(z, w) \leq \alpha\} \cup \{0\}.
  \end{equation}
  Such cone is called {\em ice cream cone} \cite[Section 4.1.9]{Cegielski2012}.

  In particular, the condition $\langle y - x, {\bf 1} \rangle \geq \beta \|y - x\|$, where ${\bf 1} := (1, \ldots, 1)$, can be equivalently written as
  \begin{equation}\label{eq1}
    \sum_{i = 1}^{N} (y_i - x_i) \geq \beta\sqrt{N \sum_{i = 1}^{N} (y_i - x_i)^2}.
  \end{equation}

  Let us consider $\mathbb{R}^2$ and $\beta := \frac{\sqrt{2}}{2} = \cos\frac{\pi}{4}$. Then from Equation~\eqref{eq1} we have $x \leqslant_K y$ if and only if
  \begin{align*}
    (y_1 - x_1) + (y_2 - x_2) & \geq \sqrt{(y_1 - x_1)^2 + (y_2 - x_2)^2} \\
    & \geq \max_i |y_i - x_i|.
  \end{align*}
  Hence, it immediately follows that $y_i \geq x_i$, that is actually $x \leqslant y$ in the order generated by the positive cone $\mathbb{R}^{N}_{+}$.

  More generally, if $\beta \leq \frac{\sqrt{2}}{2}$, then $\gamma := \beta\sqrt{2} \leq 1$ and inequality from Equation~\eqref{eq1} for $\mathbb{R}^2$ can be converted to a form useful in calculations:
  \begin{equation}
    y_1 - x_1 \geq \Lambda (y_2 - x_2) \mbox{ and } y_2 - x_2 \geq \Lambda (y_1 - x_1),
  \end{equation}
  where
  \begin{equation}
    \Lambda :=
    \begin{cases}
      \frac{-1 + \sqrt{1 - (1 - \gamma^2)^2}}{1 - \gamma^2}, & \text{for } \gamma < 1, \\
      0, & \text{for } \gamma = 1.
    \end{cases}
  \end{equation}
  From the above formula it is easy to see that, e.g. for a positive increment of $y_1-x_1$, the increment of $y_2-x_2$ may be negative, but controlled from below by a fixed coefficient $\Lambda$ equal in extreme situations $\Lambda = 0$, for $\gamma = 1$ and $\Lambda = -1$, for $\gamma = 0$ (the cone is then not {\em solid}).

\end{example}

Now assume you have given a set $S \subseteq \mathbb{R}^{N}$, then $l$ is called a {\em lowerbound to $S$ w.r.t. to $\leqslant_K$}, when
\begin{equation}
  l \leqslant_K s
\end{equation}
for all $s \in S$. Moreover, the {\em upperbound to $S$ w.r.t. to $\leqslant_K$}, $u$, is defined analogously.

An element $l$ is called the {\em infimum of $S$ w.r.t. to $\leqslant_K$}, if $l$ is a lowerbound of $S$ w.r.t. to $\leqslant_K$ and if for every other lowerbound $l^{\prime}$ of $S$ w.r.t. $\leqslant_K$ we have $l^{\prime} \leqslant_K l$.
The {\em supremum of $S$ w.r.t. to $\leqslant_K$} can be defined analogously.

\begin{remark}
  In particular, for $K := \mathbb{R}^{N}_{+}$ supremum of two always exists and satisfies
  \begin{equation}
    \sup\,\{x , y\} = \big[ \max\{x_1, y_1\}, \ldots, \max\{x_N, y_N\} \big],
  \end{equation}
  for $x, y \in \mathbb{R}^N$.
  Likewise, $\inf (x, y)$ exists and satisfies $\inf (x, y)_{i} = \min \left\{x_{i}, y_{i}\right\}$ for all $i = 1, \ldots, N$.
\end{remark}

When it is convenient, instead of the usual Euclidean norm, we convenient to consider in $\mathbb{R}^{N}$ the {\em uniform norm} (or the {\em supremum norm}) defined as $\|x\|_{\infty} := \max \left(\left|x_{1}\right|, \ldots, \left|x_{n}\right|\right)$.
This norm can be seen as a special case of {\em weighted maximum norm} (or {\em $v$-norm}) $\|\cdot\|_{v} \colon \mathbb{R}^{N} \to \mathbb{R}$ given by the formula
\begin{equation}
  \|x\|_{v} := \max_{i = 1, \ldots, N} \frac{\left|x_{i}\right|}{\left|v_{i}\right|}.
\end{equation}
where $0 \ll_{\mathbb{R}^N_{+}} v = \left(v_{1}, \ldots, v_{N}\right) \in \mathbb{R}^{N}$ is a given point \cite{feyzmahdavian2012contractive,zeng2001weighted,householder1964theory}.
\begin{remark}
  Note that for $v = {\bf 1}$, where ${\bf 1} := (1, \ldots, 1)$, we have usual sup norm $\|x\|_{\bf 1} = \|\cdot\|_{\infty}$.
\end{remark}
\begin{remark}
  Intuitively, this norm can be understood as follows: vector $v$ determines a point which, reflected by all axes of the coordinate system, will form the vertices of the cuboid.
  Then the number $\|w\|_v$, being the $v$-norm of the vector $w \in \mathbb{R}^N$, means how many times this cuboid should be enlarged so that $w$ lies on its side (facet).
  Compare Figure~\ref{fig.FG} in case of $\mathbb{R}^2$.
  \begin{figure}
    \begin{center}
      \includegraphics[width=0.9\columnwidth]{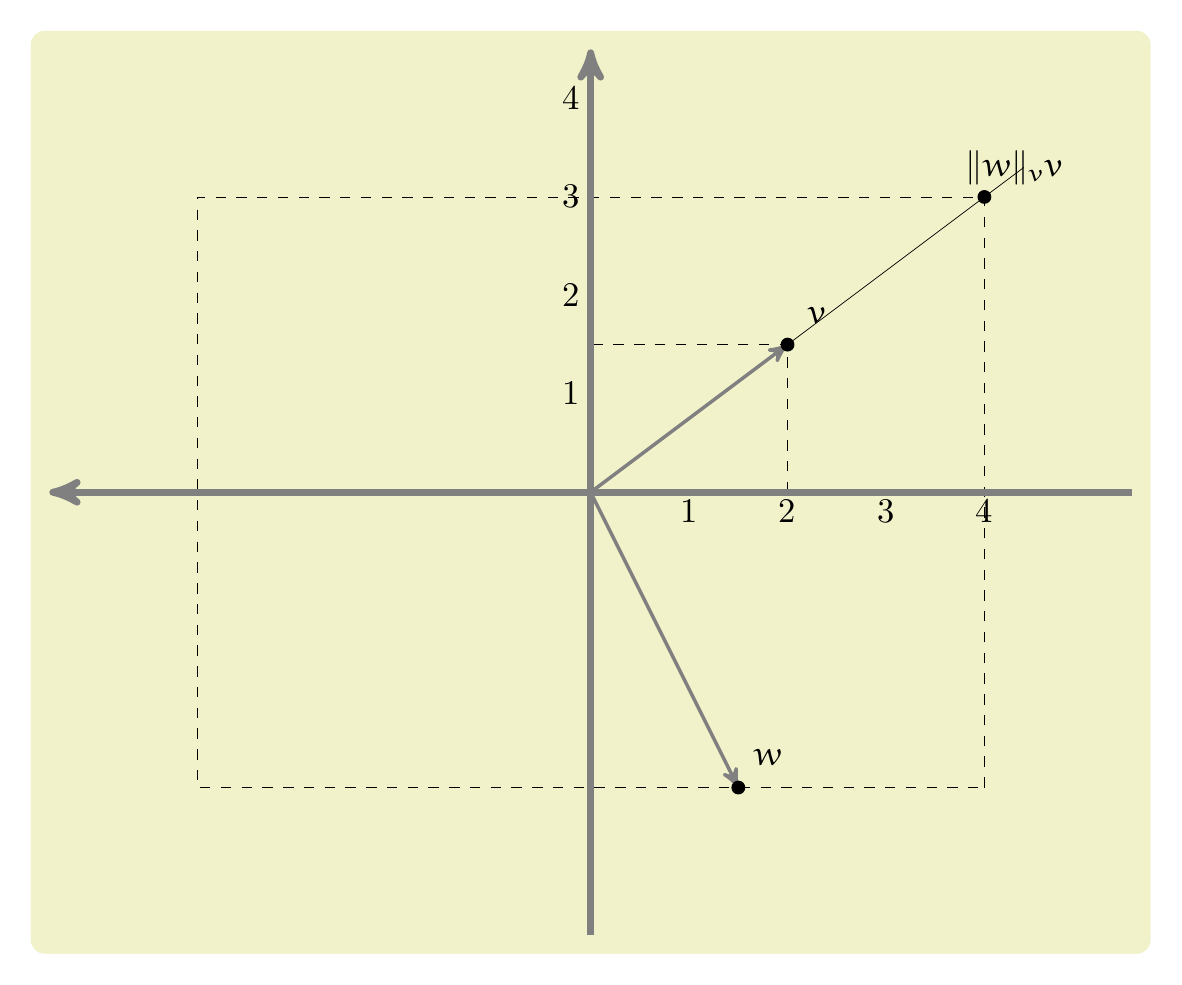}
      \caption{Intuitive meaning of weighted norm $\|w\|_v$.}
      \label{fig.FG}
    \end{center}
  \end{figure}
\end{remark}
To obtain our results, we also employ the following notions of monotonicity.
\begin{definition}
  Let $U \subseteq \mathbb{R}^{N}$. A mapping $f \colon U \to \mathbb{R}^{N}$ is called
  \begin{itemize}
    \item {\em monotone}\footnote{If we want to emphasize the cone, we say that $f$ is $K$-monotone.}, if $x <_K x^{\prime} \Rightarrow f(x) \leqslant_K f(x^{\prime})$,
    \item {\em strictly monotone}, if $x <_K x^{\prime} \Rightarrow f(x) <_K f(x^{\prime})$,
    \item {\em strongly monotone}, if $x <_K x^{\prime} \Rightarrow f(x) \ll_K f(x^{\prime})$ ($K$ need to be solid),
  \end{itemize}
  for all $x, x^{\prime} \in U$.
\end{definition}

\begin{remark}
  $K$-monotone mapping are sometimes called {\em order-preserving maps} \cite{Lemmens2012}.
\end{remark}

\begin{remark}
  A set $N \subset \mathbb{R}^{N}$ is called {\em invariant} under mapping $f$ if $f (N) \subset N$.
  It is easy to see, that if $f$ is $K$-monotone mapping it is invariant with respect to $K$, i.e. $f(K) \subset K$.
\end{remark}

\begin{definition}
  Let $0 \ll_{\mathbb{R}^N_{+}} v \in \mathbb{R}^{N}$ be fixed.
  We can generalize the above defined concepts of monotonicity. Namely, we say that $f \colon U \to \mathbb{R}^{N}$ is
  \begin{itemize}
    \item {\it norm monotone} (or {\it $K$-norm monotone}), if
      \begin{equation}
	x <_K x^{\prime} \Rightarrow \|f(x)\|_v \leqslant \|f(x^{\prime})\|_v.
      \end{equation}
  \end{itemize}
\end{definition}
The drawback of this condition is that it depends on the choice of vector $v \in \mathbb{R}^{N}$.
The below-defined assumption also generates monotonicity, but is expressed only in partial order generated by $K$.
\begin{definition}
  We say that $f \colon U \to \mathbb{R}^{N}$ is
  \begin{itemize}
    \item {\it $K$-sup-monotone}, if
      \begin{equation}
	x <_K x^{\prime} \Rightarrow f(x) \leqslant_K \sup \big\{x^{\prime}, f(x^{\prime}) \big\},
      \end{equation}
    \item {\it strictly $K$-sup-monotone}, if
      \begin{equation}
	x <_K x^{\prime} \Rightarrow f(x) <_K \sup \big\{x^{\prime}, f(x^{\prime}) \big\},
      \end{equation}
    \item {\it strongly $K$-sup-monotone}, if $K$ is solid and
      \begin{equation}
	x <_K x^{\prime} \Rightarrow f(x)\ll_K \sup \{x^{\prime}, f(x^{\prime})\}.
      \end{equation}
  \end{itemize}
  for all $x, x^{\prime} \in U$.
\end{definition}

\begin{figure}
    \begin{center}
	\includegraphics[width=0.9\columnwidth]{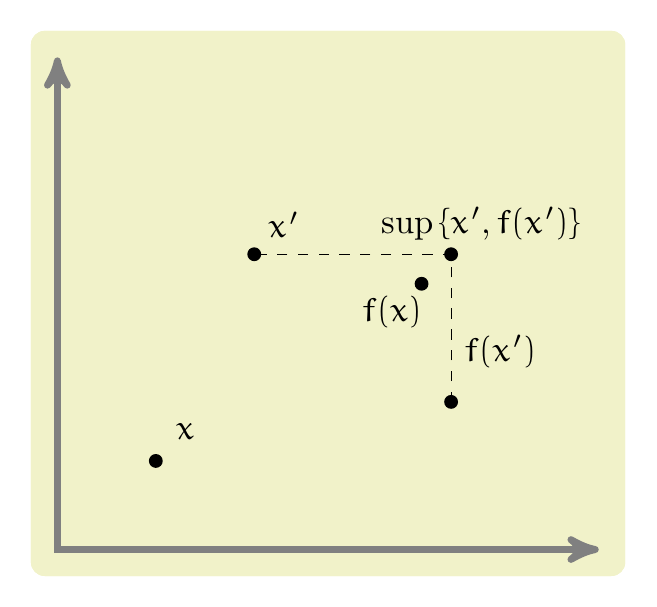}
	\caption{Intuitive meaning of sup-monotonicity.}
	\label{fig.FB}
    \end{center}
\end{figure}

\begin{remark}
  Of course, $K$-sup-monotonicity generalizes $K$-monotonicity, since if $f$ is $K$-monotone, then
  \begin{equation}
    x <_K x^{\prime} \Rightarrow f(x) \leqslant_K f(x^{\prime}) \leqslant_K \sup \big\{x^{\prime}, f(x^{\prime}) \big\}.
  \end{equation}
\end{remark}

\begin{remark}
  Let $f \colon \mathbb{R} \to \mathbb{R}$ be given by
  \begin{equation}
    f(x) := \sin(x) + x - 1.
  \end{equation}
  Then $f$ is not monotone, but is sup-monotone.
\end{remark}

In the following theorem we present an introduction to the concept and properties of topological degree from an analytic viewpoint.

\begin{theorem}
  Let $\Omega\subset\mathbb{R}^{N}$ be a bounded and open subset of $\mathbb{R}^{N}$, $f \colon \overline{\Omega} \to \mathbb{R}^{N}$ be a continuous mapping, $y_0 \in \mathbb{R}^{N}$, such that $y_0 \not\in f(\partial\Omega)$.
  Then for $(f, \Omega, y_0)$ we can assign an integer $\operatorname{deg}(f, \Omega, y_0)$ in such a way that the following properties are satisfied:
  \begin{enumerate}
    \item Existence: if $\operatorname{deg} (f, \Omega, y_0) \neq 0$, then exists ${x \in \Omega}$, such that $f(x) = y_0$.
    \item Additivity: let $\Omega_i\subset\Omega, i = 1, \ldots, m$ be a finite family of disjoint subsets of a set $\Omega$, such that $y_0 \not\in f\big(\overline{\Omega} \setminus (\cup_{i = 1}^m \Omega_i)\big)$. Then
      \begin{equation}
	\operatorname{deg} (f, \Omega, y_0) = \sum_{i = 1}^m \operatorname{deg} (f, \Omega_i, y_0).
      \end{equation}
    \item Excision: if $K = \overline{K}\subset{\overline{\Omega}}$ and $y_0 \not\in f(K) \cup f(\partial\Omega)$, then
      \begin{equation}
	\operatorname{deg} (f, \Omega, y_0) = \operatorname{deg} (f, \Omega \setminus K, y_0).
      \end{equation}
    \item Homotopy Invariance: let $f_t: \overline{\Omega}\times I \to \mathbb{R}^{N}$ be a homotopy. Let $y$ will be a mapping of interval $I$ into $\mathbb{R}^{N}$.
      If for all ${t \in I}$ we have $y(t) \not\in f_t(\partial\Omega)$, then for all $t_1, t_2 \in I$ we have 
      \begin{equation}
	\operatorname{deg} \big(f_{t_1}, \Omega, y(t_1) \big) = \operatorname{deg} \big(f_{t_2}, \Omega, y(t_2)\big)
      \end{equation}
    \item Multiplicativity: let $\Omega_1\subset \mathbb{R}^{N}$ and $\Omega_2\subset \mathbb{R}^k$ are bounded and open sets, $y_1 \in \mathbb{R}^{N}, y_2 \in \mathbb{R}^k$ and $f_1 \colon \overline{\Omega_1} \to \mathbb{R}^{N}, f_2 \colon \overline{\Omega_2} \to \mathbb{R}^k$ are continuous mappings such that $y_1 \not\in f_1(\overline{\Omega_1}), y_2 \not\in f_2(\overline{\Omega_2})$. Then
      \begin{align*}
	\operatorname{deg} \big(f_1\times f_2, \Omega_1\times\Omega_2, &(y_1, y_2)\big) = \\
	&\operatorname{deg} (f_1, \Omega_1, y_1)\cdot \operatorname{deg} (f_2, \Omega_2, y_2)
      \end{align*}
    \item Units: Let $i \colon \overline {\Omega} \to \mathbb{R}^{N}$ be the inclusion. Then
      \begin{equation}
	\operatorname{deg} (i, \Omega, 0) = 1.
      \end{equation}
  \end{enumerate}
\end{theorem}

\section{Motivation}

Let $K := \mathbb{R}_{+}^{N}$. In his seminal paper \cite{yates1995framework} Yates introduces definition of {\em interference function} $f \colon K \to \mathbb{R}^N$ which satisfies:
\begin{enumerate}
  \item {\em positivity}, i.e. $f(x) \geqslant_K 0$, for $x \geqslant_K 0$,
  \item {\em monotonicity}, i.e. if $x^{\prime} \geqslant_K x$, then $f(x^{\prime}) \geqslant_K f(x)$,
  \item {\em strict scalability}, i.e. for real number $\alpha > 1$ and $x \in K$ we have $\alpha f(x) >_K f(\alpha x)$.
\end{enumerate}
\begin{remark}
  In order to regulate interference caused by other users several models have been considered.
  Yates showed that the uplink power control problem can be reduced to finding a unique fixed point at which total transmitted power is minimized. 
\end{remark}
Let us briefly comment the assumptions. 
\begin{remark}
  Positivity means that in fact function $f$ maps $K$ into itself, i.e. $f \colon K \to K$.
\end{remark}
\begin{remark}
  Interference function have at most one fixed point, e.g. $f \colon \mathbb{R}_+ \to \mathbb{R}_+$ defined by $f(x) = x + 1$ is interference function, but does not have a fixed point.
  In order to obtain a fixed point some new condition is needed. 
\end{remark}
The main convergence result for standard interference functions \cite{feyzmahdavian2012contractive} can be summarized as follows:
\begin{theorem}
  Let $f \colon \mathbb{R}_{+}^{N} \to \mathbb{R}_{+}^{N}$ be a standard interference function and consider the iteration schema
  \begin{equation}
    x^{k+1} = f(x^{k}),
    \label{eq:yates}
  \end{equation}
  with $k$ as an index of the iteration.

  If $f$ is feasible,
  (i.e. there exists at least one point $x$ such that $f(x) \leqslant_K x$),
  then the iterates $\{x^{k}\}_{k \geq 0}$ produced by this iteration schema converge to the unique fixed-point from any initial vector $x^{0} \in \mathbb{R}_{+}^{N}$.
\end{theorem}

For the purposes of this paper we can extend the concept of scalability.
\begin{definition}
  We say that $f \colon K \to K$ is
  \begin{itemize}
    \item {\em scalable}\footnote{We sometimes write $K$-scalable, in order to refer to a specific cone.
      In the literature, this concept is also known as {\em weak scalability} \cite{piotrowski2021fixed}.}, if for real number $\alpha > 1$ and $x \in K$ we have $f(\alpha x) \leqslant_K \alpha f(x)$,
    \item {\em strictly scalable}, if for real number $\alpha > 1$ and $x \in K$ we have $f(\alpha x) <_K \alpha f(x)$,
    \item {\em strongly scalable}, if $K$ is solid and for real number $\alpha > 1$ and $x \in K$ we have $f(\alpha x) \ll_K \alpha f(x)$.
  \end{itemize}
\end{definition}


\begin{remark}
  Let $K := \mathbb{R}_{+}^{N}$.
  The notion of subhomogenity may be found frequently in the mathematical literature \cite[Section 1.4]{Lemmens2012}.
  A self-mapping $f \colon K \to K$ is called \cite{nockowska2022monotonicity}
  \begin{itemize}
    \item {\em subhomogeneous}, if $0 <_K x \Rightarrow \theta f(x) \leqslant_K f(\theta x)$,
    \item {\em strictly subhomogeneous}, if $0 <_K x \Rightarrow \theta f(x) <_K f(\theta x)$,
    \item {\em strongly subhomogeneous}, if $0 \ll_K x \Rightarrow \theta f(x) \ll_K f(\theta x)$,
  \end{itemize}
  for all $x \in K$ and $\theta \in (0, 1)$.
  It is easily to see, that (strictly, strongly) subhomogeneous mapping is (strictly, strongly) scalable mapping. 

  Indeed, these two conditions are equivalent, since, e.g., if $f$ is subhomogeneous and $\alpha > 1$, then $\frac{1}{\alpha} f(\alpha x) \ll_K f(x)$, which implies that $f(\alpha x) \ll_K \alpha f(x)$.
\end{remark}

\begin{remark}
  If $f$ is strictly scalable, then $f(0) = f(\alpha \cdot 0) <_K \alpha f(0)$, but $\alpha > 1$, so $f(0) >_K 0$.
\end{remark}

\begin{remark}\label{rem:fscalabledesceding}
  Let $f \colon K \to K$ be $K$-scalable, then for each $\alpha_{2} \geq \alpha_{1} > 1$ we have
  \begin{equation}
    \frac{1}{\alpha_{1}} f\left(\alpha_{1} x\right) \geqslant_K \frac{1}{\alpha_{2}} f\left(\alpha_{2} x\right).
  \end{equation}
  Similar inequalities can be obtained when we assume strict or strong scalability.
\end{remark}

\begin{remark}
  Scalability can be seen as a weak growth condition. Let
  \begin{equation}
    f(x) := (1 + \|x\|^{\beta}) x,
  \end{equation}
  where $\beta < 0$. Then
  \begin{equation}
    f(\alpha x) < \alpha f(x).
  \end{equation}
  In fact all concave functions are interference functions \cite[Proposition 1]{cavalcante2015elementary}.
\end{remark}

\begin{remark}
  If we apply Yates' definition of interference function to any cone $K$, then we will call such a function {\em $K$-interference function}.
\end{remark}

In \cite{piotrowski2021fixed}, the authors assumed that interference mapping can be approximated via neural network.
Such an assumption is quite convenient because closed analytical form of the map is given, and therefore it is easier to study it.
For example, we can ask what should be the conditions for weights and biases to make it scalable and monotone.

\subsection{Key concepts of neural networks}

Assume that $W \colon \mathbb{R}^{N_1} \to \mathbb{R}^{N_2}$ be a linear operator and $b \in \mathbb{R}^{N_1}$.
Let $\sigma \colon \mathbb{R}^{N_2} \to \mathbb{R}^{N_2}$ be any activation function of neural network.
Define {\em layer operator} $f \colon \mathbb{R}^{N_1} \to \mathbb{R}^{N_2}$ to be given by the formula $f(x) := \sigma(Wx + b)$.
The operation of a layer operator of a neural network can be interpreted as affine transformation of a domain of activation function.

\begin{assumption} \label{assum2}
  Let $W_i \colon \mathbb{R}^{N_{i-1}} \to \mathbb{R}^{N_{i}}$, $i = 1, \ldots, m$, be bounded linear operators and $b_i \in \mathbb{R}^{N_{i}}$, $i = 1, \ldots, m$.
  Assume that $\sigma_{i} \colon \mathbb{R}^{N_{i}} \to \mathbb{R}^{N_{i}}$, $i = 1, \ldots, m$, are continuous.
  Moreover, let us define $f_{i} \colon \mathbb{R}^{N_{i-1}} \to \mathbb{R}^{N_{i}}$ by the formula
  \begin{equation}
    f_{i}(x_{i-1}) := \sigma_{i} (W_{i} x_{i-1} + b_{i}),
  \end{equation}
  for $x_{i-1} \in \mathbb{R}^{N_{i-1}}, i = 1, \ldots, m$.
\end{assumption}
Layer operators introduced in the above assumption are building blocks of most of the neural networks used in applications.

\begin{definition}
  {\em Neural network} $f \colon \mathbb{R}^{N_0} \to \mathbb{R}^{N_m}$ is (by the very definition) a composition of layers $f_i$, i.e.
  \begin{equation}\label{eq:nn}
    f := f_m \circ \cdots \circ f_1 \colon \mathbb{R}^{N_{0}} \to \mathbb{R}^{N_{m}}.
  \end{equation}
\end{definition}

\begin{example}

  Consider neural network layer given by $f(x) = \sigma(Wx + b)$, where $\sigma: = (\sigma_1, \sigma_2)$, $\sigma_i$ is the {\em unimodal sigmoid activation function}
  \begin{equation}
    \sigma_i \colon \mathbb{R} \rightarrow \mathbb{R} \colon \xi \mapsto \frac{1}{1+e^{-\xi}}-\frac{1}{2}, \quad i = 1, 2,
  \end{equation}
  and
  \begin{equation}
    W :=
    \begin{bmatrix}
      0.95 & -0.05\\-0.05 & 0.95
    \end{bmatrix},
    \quad
    b :=
    \begin{bmatrix}
      1 \\
      -0.05 \\
    \end{bmatrix}
  \end{equation}
  Let input training data be $D := \{(0, 1), (1, 0.2)\}$. Then $f(d) >_K 0$, for $d \in D$, but, obviously, this mapping is not nonnegative.

  From this point of view, even if the data are positive, assumption that weights are positive does not have to be correct (see \cite{chorowski2014learning,ayinde2017deep} for similar remarks).

  In the above example we can not use theorem for existence of fixed points found in \cite{piotrowski2021fixed}, since the weights may be negative.

  Moreover, there is no Cybenko type theorem, which states, that positive mappings can be approximated by neural networks with positive weights.

  But $f$ leaves the following cone
  \begin{equation}
    K := C\big([1,1], \cos(\arctan(10/9)) \big) \supset \mathbb{R}^N_+. 
  \end{equation}
  invariant, so it is $K$-monotone. Moreover, $b \in K$.

\end{example}

\section{Weakening the monotonicity assumptions of interference mappings}

Although the assumption of monotonicity looks quite natural, when we check carefully the original paper, it turns out that it was introduced quite artificially.
Therefore, several authors have tried to weaken this assumption \cite{piotrowski2021fixed}.


We will use the following assumptions throughout the paper: we assume $K \subset \mathbb{R}^{N}$ is a cone such that the {\em angle of opening of cone $K$} $\alpha_K$ given by
\begin{LaTeXdescription}
  \item[(A.1)]\label{lab:A1}
    \begin{equation}
      \alpha_K := \inf_{\beta \geq 0} \{ K \subset C\big(w, \cos(\beta) \big) \}
    \end{equation}
    satisfies $\alpha_K < \pi/2$.
\end{LaTeXdescription}
Moreover, assume
\begin{LaTeXdescription}
  \item[(f.1)]\label{lab:f1} mapping $f \colon \mathbb{R}^{N} \to \mathbb{R}^{N}$ is continuous,
\end{LaTeXdescription}
From now on we assume Condition~\ref{lab:f1}, even if we do not write it explicitly.

\begin{remark}\label{remark:Sbounded}
  First, we note that if we introduce norm in $\mathbb{R}^{N}$ by $\|\cdot\| := \|\cdot\|_{\infty}$, then the condition
  \begin{LaTeXdescription}
    \item[(S.1)]\label{lab:S1} $\mathcal{S}_{f} := \{x \geqslant_K 0 \mid f(x) \leqslant_K x\}$ is bounded,
  \end{LaTeXdescription}
  implies that for some closed ball $D_R := \mathcal{D}(0, R)$ (with a correspondingly large $R > 0$) we have $f(x) \notin D_R$, for every $x$ satisfying $\|x\| = R$.
  Indeed, if not then for all $R > 0$ there would exists $x$, such that $\|x\| = R$ and $f(x) \in D_R$, i.e. $x \in \mathcal{S}_{f}$, which contradicts Assumption~\ref{lab:S1}.
\end{remark}
Feasibility condition from paper~\cite{yates1995framework} means set $\mathcal{S}_{f}$ is nonempty.
Let us strengthen this assumption a bit:
\begin{definition}
  We say that point $x \in \mathbb{R}^N$ is:
  \begin{itemize}
    \item {\em feasible} for $f$, if $x \in \mathcal{S}_{f}$ and $x \geqslant_K 0$,
    \item {\em strictly feasible} for $f$, if $x \in \mathcal{S}_{f}$ and $x >_K 0$,
    \item {\em strongly feasible} for $f$, if $x \in \mathcal{S}_{f}$ and $x \gg_K 0$.
  \end{itemize}
\end{definition}

\begin{remark}
  In other terms, we can say that $f$ is (strictly, strongly) feasible, if there exists a point (strictly, strongly) feasible for $f$.
\end{remark}

\section{Uniqueness of fixed points}

In this section, assuming the existence of a fixed point and scalability of the function, we will show that there can only be one fixed point.

\begin{theorem}\label{thm:uniquenessthm}
  Assume that $f \colon K \to \mathbb{R}^{N}$ is $K$-scalable and strongly $K$-sup-monotone. Then, if $f$ has a fixed point, then it is unique.
\end{theorem}

\begin{proof}
  Suppose there are two different fixed points $p, p^{\prime} \in \mathbb{R}^{N}_{+}$. Then the cones $p + K$ and $p^{\prime} + K$ are also different.
  Assume that $p^{\prime} + K \subset p + K$ (the situation $p + K \subset p^{\prime} + K$ is analogous) or both the sets $(p + K) \setminus (p^{\prime} + K)$ and $(p^{\prime} + K) \setminus (p + K)$ are nonempty.
  Then there is $\alpha > 1$ such that $p^{\prime} \leqslant_K \alpha p$ and $\alpha p \in \partial (p^{\prime} + K)$, that is, $\alpha$ is the smallest number satisfying $p^{\prime} \leqslant_K \alpha p$ (see Figure~\ref{fig.FD}).

  \begin{figure}
    \begin{center}
      \includegraphics[width=0.9\columnwidth]{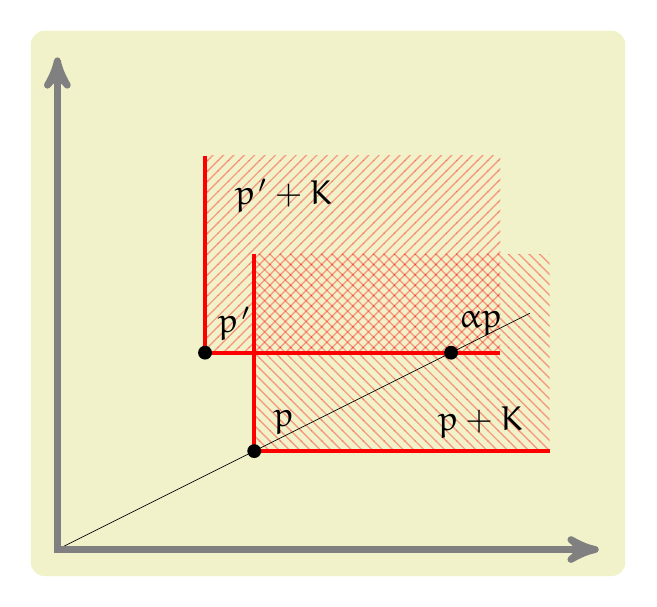}
      \caption{Existence of $\alpha$ in the proof of Theorem~\ref{thm:uniquenessthm}.}
      \label{fig.FD}
    \end{center}
  \end{figure}

  Then
  \begin{equation}
    p^{\prime} = f(p^{\prime}) \ll_K \sup \{f(\alpha p), \alpha p\} \leqslant_K \sup \{\alpha p, \alpha p\} = \alpha p,
  \end{equation}
  which contradicts the choice of $\alpha$. The proof is complete.
\end{proof}

\section{Convergence to a fixed point}

We now study existence of fixed points assuming scalability of the mapping.

Feasible point are good for as initial points for iteration schema to obtain a fixed point.

\begin{lemma}\label{lem:boundedintersection}
  If $K \subset \mathbb{R}^{N}$ is a solid cone, then for every $x \leqslant_K y$ set
  \begin{equation}
    (x + K) \cap (y - K)
  \end{equation}
  is bounded.
\end{lemma}
\begin{proof}
  By contradiction let us assume that there exists $x_n \in K$, $y_n \in -K$ such that $z_n := x + x_n = y + y_n$ and $\|z_n\| \to +\infty$ as $n \to +\infty$.
  But then $\|x_n\|, \|y_n\| \to +\infty$.

  Notice that $y_n = (x - y) + x_n$. We know that $\frac{x_n}{\|x_n\|} \in S^1\cap K$, where $S^1$ denotes unit sphere in $\mathbb{R}^{N}$.
  We can assume (up to subsequence) that $\frac{x_n}{\|x_n\|} \to x_0 \in S^1\cap K$.
  Similarly, $\frac{y_n}{\|x_n\|} \in - K$. But then
  \begin{equation}
    \left\| \frac{x_n}{\|x_n\|} - \frac{y_n}{\|x_n\|} \right\| = \left\| \frac{x - y}{\|x_n\|} \right\| \to 0,
  \end{equation}
  so $\frac{y_n}{\|x_n\|} \to x_0$. This is a contradiction because $x_0 \in K \cap (-K) = \{0\}$.
\end{proof}

\begin{lemma}\label{lem:diammeter}
  If $K \subset \mathbb{R}^{N}$ is a solid cone and $x_0 \leqslant_K x_n$, then
  \begin{equation}
    x_n \searrow_K x_0 \Rightarrow \operatorname{diam} \big( (x_0 + K) \cap (x_n - K) \big) \to 0,
  \end{equation}
  where $\searrow_K$ means that the sequence converges $K$-monotonically descending.
\end{lemma}
\begin{proof}
  Let $D_n := (x_0 + K) \cap (x_n - K)$. From the previous lemma we know that diameter of $D_n$ is bounded, for each $n$.
  Moreover, since $x_n$ tends to $x_0$ from above, we know that $\operatorname{diam}(D_n)$ is nonincreasing sequence.
  If $\operatorname{diam}(D_n) \to d > 0$, then there exists $z_n \in D_n$ such that $\|z_n - x_0\| \geq \frac{d}{2}$. So
  \begin{equation}
    z_n = x_0 + k_n = x_n + l_n,
  \end{equation}
  with $\|k_n\| \geq \frac{d}{2}$. For $n$ large enough we can assume that $\|l_n\| \geq \frac{d}{3}$, since $l_n = (x_0 - x_n) + k_n$.
  We can assume that $k_n \to k_0 \in K$. Then $l_n \to k_0$, as $n \to +\infty$, but $l_n \in -K$. This leads to a contradiction.
\end{proof}

We will need the following lemma.

\begin{lemma}\label{lem:convergenve}
  Let us assume $K \subset \mathbb{R}^{N}$ is a pointed cone.
  Every sequence $(x_n)_{n \geq 0} \subset K$ which is $K$-nonincreasing (i.e., $x_{n+1} \leqslant_K x_n$) converges to some point $x \in K$.
\end{lemma}

\begin{proof}
  %
  %
  %

  Since set $\mathcal{S}_{f}^{x_1} := \{x \geqslant_K 0 \suchthat x \leqslant_K x_1\}$ is compact (because it is closed and bounded), the sequence $(x_n)_{n \geq 1}$ has a subsequence convergent to some $x_{n_k} \to x \in \mathcal{S}_{f}^{x_1}$.
  From Lemma~\ref{lem:diammeter} proved above we know, that for $\varepsilon > 0$ there exists $\delta > 0$ such that
  \begin{equation}
    \operatorname{diam} \big( (x_0 + K) \cap (x_n - K) \big) < \varepsilon,
  \end{equation}
  if only $\|x_{n} - x_0\| < \delta$.
  There exists $k_0$ such that $\|x_{n_{k_0}} - x\| < \delta$. Thus for all $n \geq n_{k_0}$ we have
  \begin{equation}
    x_n \in \big( (x_0 + K) \cap (x_{n_{k_0}} - K) \big),
  \end{equation}
  so $\|x_{n} - x\| < \varepsilon$.
\end{proof}

\begin{lemma}\label{lemma:feasbalecontraction}
  Assume that $f \colon K \to \mathbb{R}^{N}$ be $K$-monotone and $p$ is feasible for $f$.
  Then $f^{n}(p) := (f \circ \cdots \circ f)(p)$ ($n$ times composition of $f$) is $K$-nonincreasing sequence converging to fixed point of $f$.
  If additionally $f$ is $K$-scalable, then this point is unique.
\end{lemma}
\begin{proof}
  Denote $p_0 := p$ and $p_n := f^{n}(p)$. We have $p_1 \leqslant_K p_0$. By induction, if $p_n \leqslant_K p_{n-1}$, then $p_{n+1} = f(p_n) \leqslant_K f(p_{n-1}) = p_n$.

  From Lemma~\ref{lem:convergenve} we have $p_n \searrow_K p^{*}$. But $f(p^{*}) \leqslant_K f(p_n)$, so $f(p^{*}) \leqslant_K p^{*}$.
  Moreover, $p^{*} \leqslant_K p^{n+1} = f(p_n)$, so $p^{*} \leqslant_K f(p^{*})$. Therefore, $p^{*} = f(p^{*})$.

  The last part of the theorem follows from the results in the section on the uniqueness of the fixed point.
\end{proof}

We now proceed to show that the main results in this section. To this end, we need the following definition
\begin{definition}
  A map $f \colon K \to K$ is said to be a {\em $K$-contractive interference function}, if it satisfies the following conditions:
  \begin{LaTeXdescription}
    \item[(I.1)]\label{lab:I2} $K$-monotonicity: $x \leqslant_K y$ implies $f(x) \leqslant_K f(y)$,
    \item[(I.2)]\label{lab:I3} $K$-contractivity: There exist a constant $c \in [0, 1)$ and a vector ${w} \gg_K 0$ such that 
      \begin{equation}
	f(x + \varepsilon{w}) \leqslant_K f(x) + c\varepsilon{w},
      \end{equation}
      for every $x \in K$.
  \end{LaTeXdescription}
\end{definition}

\begin{remark}
  $K := \mathbb{R}^{N}_{+}$.
  Note that a $K$-contractivity does not imply scalability (see \cite{feyzmahdavian2012contractive}). For this consider
  \begin{equation}
    f(p) :=
    \begin{cases}
      p^2 + \frac{1}{100}, & 0 \leq p \leq \frac{1}{4}, \\
      \frac{1}{2}p - \frac{1}{16} + \frac{1}{100}, & p > \frac{1}{4}. \\
    \end{cases}
  \end{equation}
  For $\alpha := 2$ and $x := \frac{1}{8}$, we have $f(\alpha \cdot x) \not <_K \alpha f(x)$.

  Similarly, scalability does not imply contractivity. For this consider $f(p) = 2p$.
\end{remark}

\begin{lemma}\label{lm0}
  There exists $\delta(K) > 0$ such that for every $0 \ll_K w$ and $v \in \mathbb{R}^{N}$, if $v, -v \leqslant_K w$, then $\|v\|_w \leq \delta(K)$.
\end{lemma}

\begin{proof}
  From Lemma~\ref{lem:boundedintersection} we know that $v, -v \in (w - K) \cap (-w + K)$, which is bounded. Therefore, $\|v\|_w \leq \delta(K)$ (comp.\ Figure~\ref{fig.FF}).

  \begin{figure}
    \begin{center}
      \includegraphics[width=0.9\columnwidth]{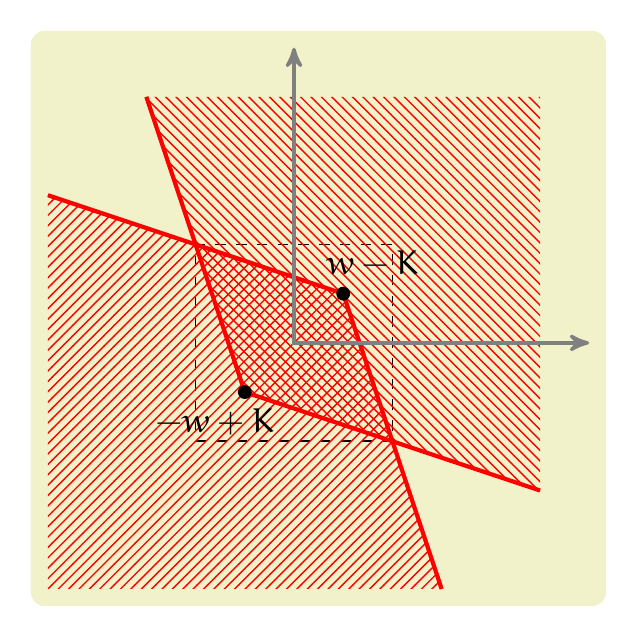}
      \caption{Intuitive representation of set $(w - K) \cap (-w + K)$.}
      \label{fig.FF}
    \end{center}
  \end{figure}

\end{proof}

\begin{lemma}\label{lm1}
  For every $\varepsilon > 0$, if $v, -v \leqslant_K \varepsilon{w}$, then $\|v\|_{w} \leq {\varepsilon}{\delta(K)}$.
\end{lemma}

\begin{proof}
  Since $v, -v \leqslant_K \varepsilon{w}$, then
  \begin{equation}
    v, -v \in (\varepsilon{w} - K) \cap (-\varepsilon{w} + K).
  \end{equation}
  Therefore,
  \begin{equation}
    \frac{1}{\varepsilon}v, \frac{-1}{\varepsilon}v \in ({w} - K) \cap (-{w} + K).
  \end{equation}
  So $\frac{1}{\varepsilon}\|v\|_{w} \leq {\delta(K)}$.
\end{proof}

\begin{lemma}\label{lm2}
  Let $\mathbb{R}^{N}_{+} \subset K$. For every ${w} \gg_K 0$ and every $v \in \mathbb{R}^{N}$ one has $v \leqslant_K \|v\|_{w}{w}$.
\end{lemma}

\begin{proof}
  The proof of this lemma can be decuded from Figure~\ref{fig.FG}.
\end{proof}

As a consequence, we obtain
\begin{lemma}
  If $f$ is $K$-contractive interference function and $c$ from its definition satisfies $c \delta(K) < 1$, then $f$ is $c\delta(K)$-contractive with respect to $\|\cdot\|_w$. 
\end{lemma}

\begin{proof}
  Take any $x \neq x^{\prime}$, $x, x^{\prime} \in K$. Then, by Lemma~\ref{lm2}, we have
  \begin{equation}
    x = x^{\prime} + (x - x^{\prime}) \leqslant_K x^{\prime} + \|x - x^{\prime}\|_{w}{w}.
  \end{equation}
  Assumptions \ref{lab:I2} and \ref{lab:I3} imply that
  \begin{equation}
    f(x) \leqslant_K f(x^{\prime} + \|x - x^{\prime}\|_{w}{w}) \leqslant_K f(x^{\prime}) + c\|x - x^{\prime}\|_{w}{w}.
  \end{equation}
  Analogously we obtain $f(x^{\prime}) \leqslant_K f(x) + c\|x - x^{\prime}\|_{w}{w}$.
  From Lemma~\ref{lm1} it follows that $\|f(x) - f(x^{\prime})\|_{w} \leq c \delta(K) \|x - x^{\prime}\|_{w}$, which finishes the proof.
\end{proof}

At that stage, we are able to derive the main result of this section, which we summarize as follows:
\begin{theorem}
  Let $\mathbb{R}^{N}_{+} \subset K$. 
  If $f$ is $K$-contractive interference function, $c \delta(K) < 1$ and $0 <_K f(0)$.
  Then $f$ has a unique fixed point $x^{*} \geqslant_K 0$, and for every initial vector $x^{0}$ the sequence $x^{k+1} := f\big(x^{k}\big)$ converges to $x^{*}$.
  Moreover, $\|x^{k} - x^{*}\|_{w} \leq \big(c\delta(K)\big)^k \|x^{0} - x^{*}\|_{w}$.
\end{theorem}


We close this section with the following important remarks.

\begin{remark}
   From a practical perspective, the above result reveal that simple iteration algorithm converge geometrically fast.
\end{remark}

\begin{remark}
  In order to obtain $x^{*} > 0$ we need assumption $0 <_K f(0)$.
\end{remark}

\section{Existence of fixed points via topological degree method}

In this section we consider the existence of fixed points of $K$-monotone neural networks.

%

\begin{theorem}\label{theorem:degreerzero}
  Assume that $\mathbb{R}^{N}_{+} \subset K$, $f$ is $K$-sup-monotone and strongly feasible. 
  Then there exits a fixed point $x^{*} = f(x^{*})$ such that $0 <_K x^{*}$.
\end{theorem}

\begin{proof}
  {\em Step 1}. We can extend the domain of $f$ to whole $\mathbb{R}^{N}$ symmetrically, i.e. $f(x) = f(z)$, for $z_i = |x_i|$, $i = 1, \ldots, n$.

  {\em Step 2}. We will show that there exits $\overline{x} > 0$ such that for every $x \in \mathbb{R}^N$:
  \begin{equation}
    x - f(x) - \overline{x} \neq 0.
  \end{equation}
  Since $\mathcal{S}_{f}$ is bounded there exists $R > 0$ such that $\mathcal{S}_{f} \subset \mathcal{B}(0, R)$ (in fact in any norm). Let $\|x\|_1 := \sum_{i = 1}^N |x_i|$.
  So there exists $\alpha > 0$ such that for every $x \in \mathbb{R}^N$ we have $\|x\| \geq \alpha \|x\|_1$.
  Let $v \in K \cap \operatorname{int} \mathbb{R}^N_{+}$ and $k := \frac{R}{\alpha \|x\|_1}$.
  Now define $\overline{x} := kv$.

  By contradiction assume that there exists $x \in \mathbb{R}^N$ such that $x - f(x) - \overline{x} = 0$. Then $x = f(x) + \overline{x} \in \operatorname{int} \mathbb{R}^N_{+}$.
  Moreover, $f(x) \leqslant_K x$, since $\overline{x} \in K$, so $x \in \mathcal{S}_{f}$. Therefore, $\|x\| < R$. Since $x = f(x) + \overline{x} \geqslant_K \overline{x}$, then
  \begin{equation}
    \|x\| \geq \alpha \|x\|_1 \geq \alpha \|k v\|_1 = \alpha k \|v\|_1 \geq R.
  \end{equation}
  This leads to a contradiction with the choice of $x$.

  {\em Step 3}. We will show, that for each $x \in \partial \mathcal{B}(0, R)$ and $t \in [0, 1]$ we have $x - f(x) - t \overline{x} \neq 0$.Suppose that there exists $x \in \partial \mathcal{B}(0, R)$ and $t \in [0, 1]$ such that 
  \begin{equation}
    x = f(x) + t \overline{x} \in \operatorname{int} \mathbb{R}^N_{+}.
  \end{equation}
  Then $f(x) \leqslant_K x$, so $\|x\| < R$; a contradiction.

  {\em Step 4}. Consider the homotopy $H_{t} := \mathbb{I} - f - t \overline{x}$.
  For $\|x\| = R$ we have $H_{t}(x) \neq 0$, for all $t \in [0, 1]$ which means that the degree of $H_{t}$ relative to $D_R := \mathcal{D}(0,R)$ at 0 is well defined.
  By homotopy invariance property of topological degree we know that
  \begin{align*}
    \operatorname{deg}\left(\mathbb{I} - f, D_R, 0\right) & = \operatorname{deg}\left(H_{0}, D_R, 0\right) = \\
    \operatorname{deg}\left(H_{1}, D_R, 0\right) & = d\left(\mathbb{I} - f - \overline{x}, D_R, 0\right).
  \end{align*}
  Since $\mathcal{S}_{f} = \varnothing$, then there is no $x \in \overline{D}_R$, for which $x - f(x) - \overline{x} = 0$.
  Hence, by the contraposition of the existence property we have
  \begin{equation}
    \operatorname{deg}\left(\mathbb{I} - f, D_R, 0\right) = 0.
  \end{equation}

  {\em Step 5}.

  Define $D_{+} := \{z \leqslant_K 0 \suchthat z \leqslant_K x^{\prime}\}$ and
  \begin{equation}
    D := \{z = (z_1, \ldots, z_n) \in \mathbb{R}^{N} \suchthat (|z_1|, \ldots, |z_n|) \in D_{+}\}.
  \end{equation}
  By $K$-sup-monotonicity $D$ is convex, bounded, closed and $0 \in \operatorname{int} D$. It is enough to show that for $x \in D_{+}$ we have $f(x) \in D_{+}$. 
  We need to check this property only for $x \in \partial D_{+} \cap \operatorname{int} \mathbb{R}^N_{+}$. For such $s$ we have $x \leqslant_K x^{\prime}$. By $K$-sup-monotonicity we have
  \begin{equation}
    f(x) \leqslant_K \big\{x^{\prime}, f(x^{\prime}) \big\} = x^{\prime},
  \end{equation}
  because $f(x^{\prime}) \leqslant_K x^{\prime}$.

  {\em Step 6}.

  If $x = f(x)$, for some $x \in \partial D_{+} \cap \operatorname{int} \mathbb{R}^N_{+}$, then we have the thesis of the theorem.
  If $x \neq f(x)$, then in particular $x \neq f(x)$, for $x \in \partial D$. Consider homotopy
  \begin{equation}
    H_t(x) := \mathbb{I} - tf(x).
  \end{equation}
  If $x = tf(x)$ on $\partial D$, then $x = tf(x) \ll_K f(x) \leqslant_K x$ which is a contradiction. By the homotopy property of topological degree we have
  \begin{equation}
    \operatorname{deg}(H_0, D, 0) = \operatorname{deg}(H_1, D, 0) = \operatorname{deg}(\mathbb{I}-f, D, 0).
  \end{equation}
  By the Additivity property there must exists a fixed point in $(\mathcal{B} \setminus D) \cap \mathbb{R}^N_{+}$.

\end{proof}

\begin{remark}
  In order to show the above theorem we do not needed the monotonicity of the mapping $f$, but boundedness of $\mathcal{S}_{f}$ was crucial.

  Assumption~\ref{lab:S1} replaces scalability, but may be hard to check.
\end{remark}

As an immediate consequence of Theorem above, we obtain that existence of fixed points if function $f$ is monotone.
More precisely, we have

\begin{corollary}
  If $f$ is $K$-monotone, then the theorem follows. In particular, theorem is valid for $K = \mathbb{R}^N_{+}$ \cite{persson2006fixed}.
\end{corollary}

\begin{corollary}
  Instead of monotonicity, we can take $x^{\prime} \in \mathcal{S}_{f}$, $x^{\prime} \gg_K 0$ and assume that for every $x, \overline{x} \in \mathbb{R}^N$ we have
  \begin{equation}
    x \leqslant_K \overline{x} \Rightarrow \|f(x)\|_{x^{\prime}} \leq \|f(\overline{x})\|_{x^{\prime}}.
  \end{equation}
  Then there exits a fixed point of $f$.
\end{corollary}

\begin{proof}
  It is enough to note that on boundary of $D$ we have $x \leqslant_K x^{\prime}$, so
  \begin{equation}
    \|f(x)\|_{x^{\prime}} \leq \|f(\overline{x})\|_{x^{\prime}} \leq 1,
  \end{equation}
  so $f(x) \in D$.
\end{proof}

We close this section by noticing that sup-monotonicity is enough for existence of fixed points.
\begin{corollary}
  Instead of monotonicity we assume that for each $x, x^{\prime} \in \mathbb{R}^N_{+}$ such that $x < f(x)$ we have 
  \begin{equation}
    x \leqslant_K x^{\prime} \Rightarrow f(x) \leqslant_K \sup \big\{x^{\prime}, f(x^{\prime}) \big\}.
  \end{equation}
  Then there exits a fixed point of $f$.
\end{corollary}

\begin{proof}
  In the proof of above theorem we assume that $x \neq f(x)$ on $\partial D$ and $x = tf(x)$, for $t \in (0, 1)$. Thus $f(x) > x$, so by above assumption
  \begin{equation}
    f(x) \leqslant_K \sup \big\{x^{\prime}, f(x^{\prime}) \big\} = x^{\prime},
  \end{equation}
  so $f(x) \in D$.
\end{proof}

\section{Existence of several fixed points}

We now proceed to study existence of several fixed points.
In particular, the next theorem proves that, with an appropriate arrangement of points and inequalities satisfied, there exist at least three fixed points of a mapping.



\begin{theorem}
  Let $f \colon K \to \mathbb{R}^{N}$ be a continuous $K$-monotone map.
  Assume that the set $\mathcal{S}_{f} = \{x \geqslant_K 0 \suchthat f(x) \leqslant_K x\}$ is bounded, there are two points $x^{\prime}, x^{\second} >_K 0$ in $\mathcal{S}_{f}$ such that $f(x^{\prime}) \ll_K x^{\prime}$ and $f(x^{\second}) \ll_K x^{\second}$, and a point $x^{\prime} \leqslant_K x \leqslant_K x^{\second}$ such that $x \ll_K f(x)$.
  Then there exist at least three fixed points of $f$.
\end{theorem}

\begin{proof}
  We can check that 
  \begin{equation}
    \operatorname{deg}\left(\mathbb{I}-f, \operatorname{int} D, 0\right) = 1,
  \end{equation}
  where $D := \{z = (z_1, \ldots, z_n) \in \mathbb{R}^{N} \suchthat (|z_1|, \ldots, |z_n|) \in D_{+}\}$ and $D_{+} := \{z \geqslant_K 0 \suchthat z \leqslant_K x^{\prime}\}$.
  This implies that at least one fixed point is in $\operatorname{int} D_{+}$.
  Note that the above computation of the degree is possible because the $K$-monotonicity and the assumption $f(x^{\prime}) \ll_K x^{\prime}$ imply that there are no fixed points of $f$ on the boundary of $D$.

  We also can check that 
  \begin{equation}
    \operatorname{deg}\left(\mathbb{I} - f, \operatorname{int} D^{\prime}, 0\right) = 1,
  \end{equation}
  where $D^{\prime} := \{z \geqslant_K 0 \suchthat x \leqslant_K z \leqslant_K x^{\second}\}$.
  Here also assumptions on $x, x^{\second}$ and $K$-monotonicity play a role.

  The additivity property of the degree implies that $\operatorname{deg} \left(\mathbb{I} - f, \mathcal{D}(0,R) \setminus (D\cup D^{\prime}), 0\right) = -2$.
  Hence, there is a fixed point of $f$ in $\mathcal{D}(0,R) \setminus (D\cup D^{\prime})$. This completes the proof.
\end{proof}

\begin{theorem}\label{thm5}
  Let $f \colon K \to \mathbb{R}^{N}$ be a continuous mapping.
  Suppose there is $x^{\prime} > 0$ in set $\mathcal{S}_{f} = \{x \geqslant_K 0 \suchthat f(x) \leqslant_K x\}$ and $f$ is $K$-monotonic with respect to the cone $K = C({w},\beta)$, i.e. $x \leqslant_K y$ implies $f(x) \leqslant_K f(y)$.
  Then there exists a fixed point of $f$.
\end{theorem}

\begin{proof}

  The boundedness of the set $\mathcal{S}_{f}$ implies the existence of a number $R > 0$ such that $\mathcal{S}_{f} \subset \mathcal{B}(0,R)$ (no matter what norm).
  Let us fix the norm $\|\cdot\|$. It is equivalent to the norm $\|\cdot\|_1$, so there is a $\alpha > 0$ such that $\|x\| \geq \alpha \|x\|_1$, for each $x \in \mathbb{R}^{N}$.
  We extend mapping $f$ on whole $\mathbb{R}^{N}$ as in \cite{persson2006fixed}. Let us take a point $\bar x := k{w}$ such that $k \geq \frac{R}{\alpha \|{w}\|_1}$.
  We will show that $x - f(x) - k{w} \neq 0$, for every $x \in \mathbb{R}^{N}$.

  Suppose there is some $x \in \mathbb{R}^{N}$ with $x = f(x) + k{w}$. Then $f(x) \leqslant_K x$, which implies that $x \in \mathcal{S}_{f}$ and, consequently, $\|x\| < R$.
  We also notice that $k{w} \leqslant_K x$ because $0 \leqslant_K f(x)$. Therefore
  \begin{equation}
    \|x\| \geq \alpha \|x\|_1 \geq \alpha \|k{w}\|_1 = \alpha k\|{w}\|_1 \geq \alpha \|{w}\|_1\frac{R}{\alpha \|{w}\|_1} = R;
  \end{equation}
  a contradiction.

  Note that for every $x$ with $\|x\| = R$ one has $x - f(x) - tk{w} \neq 0$, for every $t \in [0, 1]$.
  Indeed, assuming that $x = f(x) + tk{w}$ we obtain, as above, $f(x) \leqslant_K x$ which implies $x \in \mathcal{S}_{f}$ and, consequently, $\|x\| < R$.

  Now we use a homotopy $H_t := \mathbb{I} - f - tk{w}$ and obtain that $\operatorname{deg}\left(\mathbb{I} - f, \mathcal{B}(0,R), 0\right) = 0$.

  Consider the sets $D_{+} := \{z \geqslant_K 0 \suchthat z \leqslant_K x^{\prime}\}$ and
  \begin{equation}
    D := \{z = (z_1, \ldots, z_n) \in \mathbb{R}^{N} \suchthat (|z_1|, \ldots, |z_n|) \in D_{+}\}.
  \end{equation}
  Notice that $D_{+} = (x^{\prime} - K) \cap \mathbb{R}^{N}_{+}$ is a closed bounded set, and so is $D$. Moreover, $0$ is in the interior of $D$.

  We will show that $D$ is invariant with respect to $f$. Obviously, it is enough to prove that $f(x) \in D$, for $0 \leqslant_K x \in D$.
  To do this, let us assume that $0 \leqslant_K x \in \partial D$.
  Then $x \leqslant_K x^{\prime}$ (because $x - x^{\prime} \in -K$ implies that $x^{\prime} - x \in K$), and from the $K$-monotonicity we obtain $f(x) \leqslant_K f(x^{\prime})$.
  This implies that $f(x) \in f(x^{\prime}) - K$. Hence, $f(x) \in D$.

  Now, in a standard way we show that
  \begin{equation}
    \operatorname{deg}\left(\mathbb{I}-f, \operatorname{int} D, 0\right) = 1.
  \end{equation}

  The additivity property of the degree implies that $\operatorname{deg}\left(\mathbb{I} - f, \mathcal{D}(0,R) \setminus D, 0\right) = -1$.
  Hence, there is a fixed point of $f$ in $\mathcal{D}(0,R) \setminus D$. This completes the proof.
\end{proof}

Theorem~\ref{thm5} follows from the more general Theorem~\ref{thm6}.
\begin{theorem}
  \label{thm6}
  Let $f \colon \mathbb{R}^{N}_{+} \to \mathbb{R}^{N}_{+}$ be continuous and $K$-monotone for $K = C({w}, \beta)$.
  Assume that there is $x^{\prime} \gg_K 0$ in the set $\mathcal{S}_{f} = \{x \geqslant_K 0 \suchthat f(x) \leqslant_K x\}$, and $x^{\second} \gg_K x^{\prime}$ such that 
  \begin{equation}\label{eq11}
    x^{\second} \not\in f(x) + K,
  \end{equation}
  for any $x \in T := \{z \geqslant_K 0 \suchthat x^{\second} \in \partial (z + K)\}$.
  Then there exists a fixed point of $f$ in the set $\{z \geqslant_K 0 \suchthat z \leqslant_K x^{\second} \mbox{ and } z \not \ll_K x^{\prime}\}$.
\end{theorem}

Let us define
\begin{equation}
  D(x) := \{(z_1, \ldots, z_n) \in \mathbb{R}^{N} \suchthat (|z_1|, \ldots, |z_n|) \leqslant_K x\},
\end{equation}
for $x \in \mathbb{R}^{N}$.

Note that Assumption~\eqref{eq11} can be described in terms of a $D(x^{\second})$-norm 
\begin{equation}
  \|x\|_{D(x^{\second})} := \inf\{\alpha > 0 \suchthat \alpha x \in D(x^{\second})\},
\end{equation}
as
\begin{equation}\label{eq11bis}
  \|x\|_{D(x^{\second})} = 1 \Rightarrow \|f(x)\|_{D(x^{\second})} > 1.
\end{equation}

\begin{proof}
  We extend $f$ onto $\mathbb{R}^{N}$ as follows $\bar f(z_1, \ldots, z_n) := f(|z_1|, \ldots, |z_n|)$.
  We will also denote this extended map by $f$. Take a point $\bar x := k{w}$ with $k$ so big that $\bar x \not\in D(x^{\second})$.
  It is easy to see that $x - f(x) - k{w} \neq 0$, for every $x \in D(x^{\second})$.

  Indeed, if $x = f(x) + k{w}$, for some $x \in D(x^{\second})$, then $k{w} \leqslant_K x$ which implies that $k{w} \in D(x^{\second})$ but this is impossible. 

  Note that for every $x \in T$ one has $x - f(x) - tk{w} \neq 0$, for every $t \in [0, 1]$.
  Indeed, assuming that, for some $x \in T$, we have $x = f(x) + tk{w}$, we obtain that $f(x) \leqslant_K x$ which contradicts Condition~\eqref{eq11}.

  Now we use a homotopy $H_t := \mathbb{I}-f-tk{w}$ and obtain that $\operatorname{deg}\left(\mathbb{I} - f, \mathcal{B}(0,R), 0\right) = 0$.

  Consider the set $D(x^{\prime})$. We will show that it is invariant with respect to $f$.
  Obviously, it is enough to prove that $f(x) \in D(x^{\prime})$, for $0 \leqslant_K x \in D(x^{\prime})$.
  To do this, let us assume that $0 \leqslant_K x \in \partial D(x^{\prime})$.
  Then $x \leqslant_K x^{\prime}$ (because $x - x^{\prime} \in -K$ implies that $x^{\prime} - x \in K$), and from the $K$-monotonicity we obtain $f(x) \leqslant_K f(x^{\prime})$.
  This implies that $f(x) \in f(x^{\prime}) - K$. Hence, $f(x) \in D(x^{\prime})$.

  Now, if we assume that there is no fixed point on the boundary of $D(x^{\prime})$ (since otherwise the proof would be finished), in a standard way we show that
  \begin{equation}
    \operatorname{deg}\left(\mathbb{I} - f, \operatorname{int} D(x^{\prime}), 0\right) = 1.
  \end{equation}

  The additivity property of the degree implies that $\operatorname{deg}\left(\mathbb{I} - f, \operatorname{int} D(x^{\second}) \setminus D(x^{\prime}), 0\right) = -1$.
  Hence, there is a fixed point of $f$ in $\operatorname{int} D(x^{\second}) \setminus D(x^{\prime})$. This completes the proof.
\end{proof}

\begin{remark}
  Notice that a boundedness of the set $\mathcal{S}_{f}$ in Theorem~\ref{thm5} implies the existence of the point $x^{\second}$ mentioned in Theorem~\ref{thm6}.
\end{remark}

This generalization of Theorem~\ref{thm5} present in Theorem~\ref{thm6} has nice consequences in neural networks as we can see below.

\begin{example}\label{ex:example3}
  For simplicity let us consider a one dimensional case. Define a linear map $W(x) := 10x$ and an sigmoid activation function $\sigma(x) = \frac{1}{1 + \exp(-x)}$.
  Put $f(x) := \sigma(W(x) - 4)$. Plot of this mapping can be seen in Figure~\ref{fig.ex1}.
  \begin{figure}
    \begin{center}
      \includegraphics[width=0.9\columnwidth]{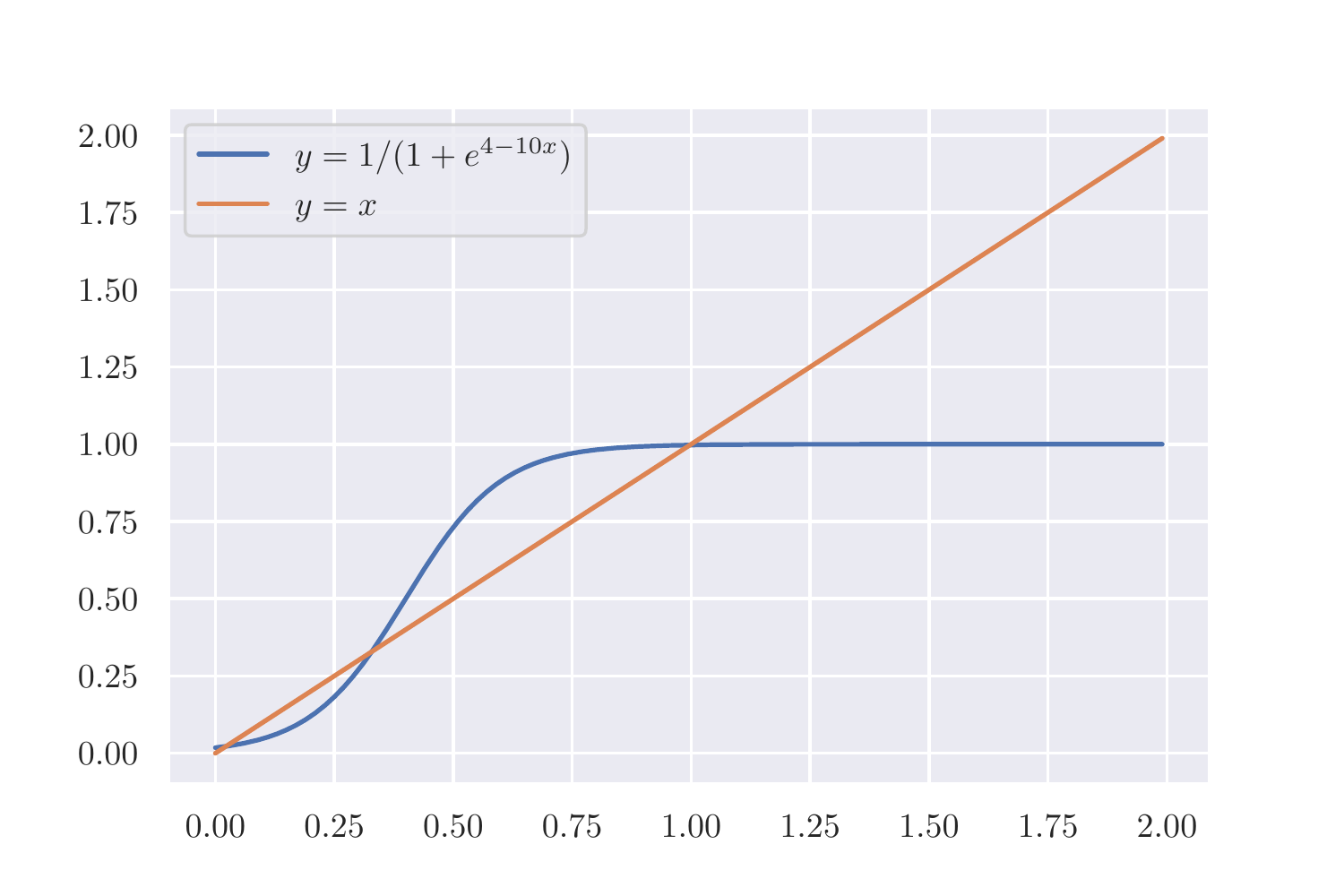}
      \caption{Plot of the function $f$ in Example~\ref{ex:example3}.}
      \label{fig.ex1}
    \end{center}
  \end{figure}

  It is easy to check that $f(0.6) \approx 0.88 > 0.6 =: x^{\second}$ and $f(0.3) \approx 0.27 < 0.3 =: x^{\prime}$.
  Therefore, there exists a fixed point of $f$ in interval $(x^{\prime}, x^{\second})$.

  Obviously, in the above situation Darboux theorem is enough, but in higher dimension we need to use the topological degree technique.

  Notice that mapping $f$ is not scalable and the set $\mathcal{S}_{f}$ is not bounded. Moreover, $\|W\| > 1$.

  In similar fashion we can construct an example in $\mathbb{R}^{N}$ with $K := \mathbb{R}^N_{+}$.
  For instance, we can take some $W$ with positive entries, some negative bias $b \in \mathbb{R}^{N}$ and a sigmoidal activation function $\sigma \colon \mathbb{R} \to \mathbb{R}$.
  Then we can define neural network layer $f(x) := \bar\sigma(W(x) + b)$, where $\bar\sigma (x) := (\sigma(x_1), \ldots, \sigma(x_n))$.
\end{example}



\begin{theorem}\cite[Proposition 7]{persson2006fixed}
  \label{thm8}
  Let $f \colon \mathbb{R}^{N} \to \mathbb{R}^{N}$ be continuous and $K$-monotone, the set $\mathcal{S}_{f} = \{x \geqslant_K 0 \suchthat f(x) \leqslant_K x\}$ is bounded and there exists $x^{\prime}, x^{\second} \in \mathcal{S}_{f}$, $x^{\second} \ll_K x^{\prime}$.
  Then there exist at least two fixed points $\overline{x}, \widetilde{x}$ of $f$ such that $\overline{x} < \widetilde{x}$.
\end{theorem}

\begin{proof}
  Similarly to the above (lemmata \ref{lem:boundedintersection}--\ref{lemma:feasbalecontraction}), due to the boundedness of the set $\mathcal{S}_{f}$, we obtain for $x^{\second}$ the non-increasing sequence $x^n := f^n(x^{\second})$ converging to a point $\overline{x} \in \mathcal{S}_{f}$.

  Let $z := x - \overline{x}$ and $z^{\prime} := x^{\prime} - \overline{x} \gg_K 0$.
  Define $g(z) := g(x - \overline{x}) = f(x) - \overline{x}$. Then if $z \leqslant_K z_1$, we get $g(z) \leqslant_K g(z_1)$.

  Morover, $g(z^{\prime}) = f(x^{\prime}) - \overline{x} \leqslant_K x^{\prime} - \overline{x} = z^{\prime} \in \mathcal{S}_{g} := \{v \suchthat g(v) \leqslant_K v\}$.
  Therefore, there exists a fixed point $\widetilde{z} = g(\widetilde{z}) > 0$, but then $\widetilde{z} = \widetilde{x} - \overline{x}$ and so $\widetilde{x} = f(\widetilde{x})$ and $\widetilde{x} > \overline{x}$.
\end{proof}

\begin{theorem}
  \label{thm9}
  Let $f \colon \mathbb{R}^{N}_+ \to \mathbb{R}^{N}_+$ be continuous and $K$-sup-monotone, the set $\mathcal{S}_{f} = \{x \geqslant_K 0 \suchthat f(x) \leqslant_K x\}$ is bounded and there exists $x^{\prime}, x^{\second} \in \mathcal{S}_{f}$, $0 \ll_K x^{\second} \ll_K x^{\prime}$.
  Then there exist at least two fixed points $0 \leqslant_K \overline{x}$, $0 \leqslant_K \widetilde{x}$ of $f$.
\end{theorem}

\begin{proof}
  We know that $\operatorname{deg}(\mathbb{I} - f, D(x^{\second}), 0) = 1$, so there exists a fixed point in $D(x^{\second})$.
  We know also that there exits a fixed point in $\mathcal{B}(0, R) \setminus D(x^{\prime})$.
\end{proof}

\section{Existence of fixed points via guiding function approach}


In this section, we assume that $K := \mathbb{R}^{N}_{+}$. Let us introduce new type of assumption:
\begin{LaTeXdescription}
  \item[(G)]\label{lab:Z4} $\forall_{x, x^{\prime} \geqslant_K 0}\left(\left\|x^{\prime}\right\|_{1} = \|x\|_{1} \Rightarrow\left\langle f(x) - x, f\left(x^{\prime}\right)-x^{\prime}\right\rangle \geq 0\right)$,
    where $\|y\|_{1} := \sum_{i = 1}^{N}\left|y_{i}\right|$, for $y \in \mathbb{R}^{N}$.
\end{LaTeXdescription}
This assumption is similar to the guiding function assumption. 
\begin{remark}
  In Figure~\ref{fig.assumptionz4} we illustrate the Assumption~\ref{lab:Z4}.

  \begin{figure}
    \begin{center}
      \includegraphics[width=0.9\columnwidth]{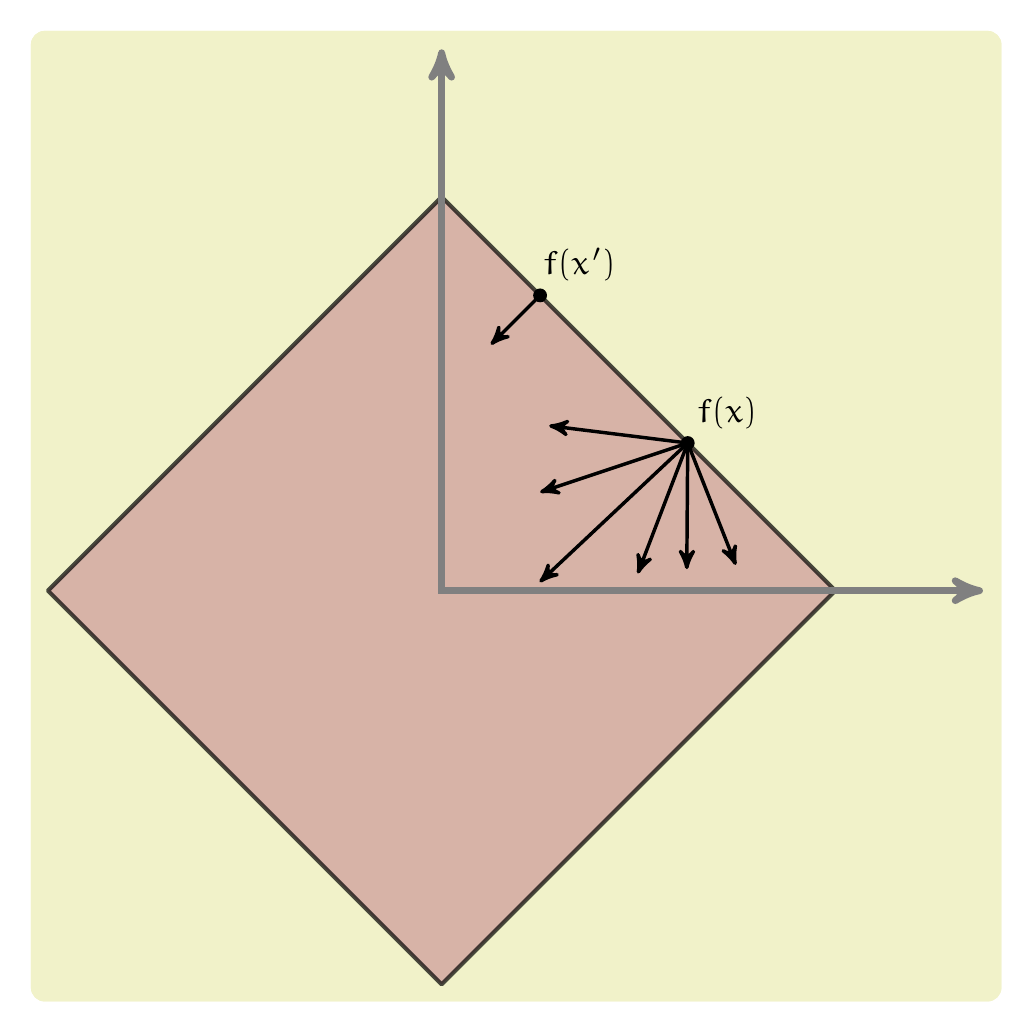}
      \caption{Illustration to Assumption~\ref{lab:Z4}.}
      \label{fig.assumptionz4}
    \end{center}
  \end{figure}

\end{remark}
It is possible to show the existence of a positive fixed point assuming the boundedness of the set $\mathcal{S}_{f}$ and the above assumption:
\begin{theorem}
  Let us assume \ref{lab:S1}, \ref{lab:Z4} and that there exists $x^{\prime} > 0$ such that $x^{\prime}-f\left(x^{\prime}\right) = \lambda {\bf 1}$, for some $\lambda \geq 0$.
  Then there exists a fixed point of $f$.
\end{theorem}

\noindent {\em Sketch of the proof.}
Indeed, on the boundary of the ball $\mathcal{B}\left(0, \left\|x^{\prime}\right\|_{{\bf 1}}\right)$ there exists point $x^{\prime}$, and $f\left(x^{\prime}\right) - x^{\prime}$ is an internal vector normal to this boundary, so for every other point from the boundary we have $f(x) - x \in T_{D_R}(x)$, and so $f(x) \in D_{1}$.\hfill$\square$

%

\begin{theorem}
  Assume \ref{lab:S1}. Let us assume that there exists $x^{\prime} \gg_K 0$ in $\mathcal{S}_{f}$ and assume
  \begin{equation}\label{eq:Z2}
    \begin{split}
      \forall_{x \geq_K 0} & (f(x) \neq x \leqslant_K f(x) \Rightarrow \\
      & \forall_{x^{\prime} \geq_K 0} \left(x^{\prime} \geq_K x \Rightarrow \left\|f\left(x^{\prime}\right)\right\|_{v} \geq \|f(x)\|_{v}\right)),
    \end{split}
  \end{equation}
  for $v = x^{\prime}$. Then there exists a fixed point of $f$.
\end{theorem}

\noindent {\em Sketch of the proof.}
Indeed, taking $x$ on the boundary of $D_{1}$ and $i$ such that $x_{i} = x_{i}^{\prime}$ we get (assuming $x = t f(x)$ with $t \in [0, 1)$) inequality $x_{i} < f\left(x_{i}\right)$, what gives $\|f(x)\|_{x^{\prime}} > 1$.
From Assumption~\eqref{eq:Z2} we get $\left\|f\left(x^{\prime}\right)\right\|_{x^{\prime}} \geq \|f(x)\|_{x^{\prime}} > 1$; a contradiction.\hfill$\square$


Note that with $x^{\prime}$ lying on the "main diagonal" we have the following result:
\begin{theorem}
  Assume \ref{lab:S1}. Suppose there is a number $\lambda > 0$ such that $f(\lambda {\bf 1}) \leqslant_K \lambda {\bf 1}$ and Condition~\eqref{eq:Z2} is satisfied for $v = {\bf 1}$, so
  \begin{align*}
    \forall_{x \geqslant_K 0}\{ f(x) \neq x & \leqslant_K f(x) \Rightarrow \\
    & \forall_{x^{\prime} \geqslant_K 0}(x^{\prime} \geqslant_K x \Rightarrow
    \|f(x^{\prime})\|_{\infty} \geq \|f(x)\|_{\infty}) \}.
  \end{align*}
  Then there exists a positive fixed point of $f$.
\end{theorem}
\begin{remark}
  Assumption $f(\lambda {\bf 1}) \leqslant_K \lambda {\bf 1}$ means that at the vertex of some square at which mapping $f$ value is directed to its interior.
\end{remark}


Assumption~\ref{lab:Z4} can be generalized to the following form:
\begin{LaTeXdescription}
  \item[(G2)]\label{lab:Z5} there exists $\gamma \in [-\pi/2, \pi/2]$ such that for $x, x^{\prime} \in \mathbb{R}^{N}_{+}$ with $\|x\|_1 = \|x^{\prime}\|_1$ we have
    \begin{equation*}
      \langle f(x) - x, f(x^{\prime}) - x^{\prime} \rangle \geq \cos(\gamma) \|f(x) - x \|_2 \cdot \| f(x^{\prime}) - x^{\prime} \|_2.
    \end{equation*}
\end{LaTeXdescription}

The above facts enable us to show 
\begin{theorem}
  Assume \ref{lab:S1}, \ref{lab:Z5}.
  Then there exists a positive fixed point of $f$.
\end{theorem}

\begin{theorem}
  Assume \ref{lab:S1}, \ref{lab:Z5} and suppose there exists a number $\lambda > 0$ such that $f(\lambda {\bf 1}) \leqslant_K \lambda {\bf 1}$, then
  \begin{equation}
    \langle f(x) - x, {\bf 1} \rangle \leq \left(\frac{\pi}{2} - \gamma \right) \|f(x) - x \|_2,
  \end{equation}
  where $\|\cdot\|_2$ denotes the 2-norm in $\mathbb{R}^{N}$.
  Then there exists a positive fixed point of $f$.
\end{theorem}
\begin{remark}
  The above assumption means that there exists $x^{\prime}$ such that $f(x^{\prime}) = x^{\prime} + \lambda {\bf 1}$.
  In \cite[Chapter 6]{Lemmens2012} such $x'$ is called {\em additive eigenvalue of $f$}.
\end{remark}


\section{Examples}

To illustrate the results obtained in the previous sections in a concrete application, we present two examples.

The following two lists provide examples of typical activation functions $\sigma \colon \mathbb{R}_{+} \to \mathbb{R}_{+}$ used in theory of neural networks:
\begin{itemize}
  \item (sigmoid) $\xi \mapsto \frac{1}{1+\exp (-\xi)}$,
  \item (capped ReLU) $\xi \mapsto \min \{\xi, \beta\}$, for $\beta > 0$,
  \item (saturated linear) $\xi \mapsto\left\{\begin{array}{ll}1, & \text { if } \xi > 1, \\ \xi, & \text { if } 0 \leq \xi \leq 1, \end{array}\right.$
    \item (inverse square root unit) $\xi \mapsto \frac{\xi}{\sqrt{1+\xi^{2}}}$,
    \item (arctangent) $\xi \mapsto(2 / \pi) \arctan \xi$,
    \item (hyperbolic tangent) $\xi \mapsto \tanh \xi$,
    \item (inverse hyperbolic sine) $\xi \mapsto \operatorname{arcsinh} \xi$,
    \item (Elliot) $\xi \mapsto \frac{\xi}{1+\xi}$,
    \item (logarithmic) $\xi \mapsto \log (1+\xi)$,
    \item (Swish function) $\xi \mapsto \frac{\xi}{1+\exp (-\xi)}$,
    \item (Mish function) $\xi \mapsto \xi \tanh (\log (1+\exp (\xi)))$.
\end{itemize}

\begin{example}
  (An example of a mapping that satisfies weaker than monotonicity condition and is not monotonic) We will show an example of a function satisfying Condition~\eqref{eq:Z2}, but which is not monotone.

  Let us consider zigzag function $\alpha \colon \mathbb{R} \to \mathbb{R}$ defined by
  \begin{equation}
    \alpha(t) :=
    \left\{\begin{array}{ll}
      t - 2n, & \text { if } t \in [2 n, 2n + 1), \\
      -t + 2n + 2, & \text { if } t \in [2n + 1, 2n + 2),
    \end{array}\right.
  \end{equation}
  for $n \in \mathbb{Z}$.
  \begin{figure}
    \begin{center}
      \includegraphics[width=0.45\textwidth,trim=0cm 3.5cm 0cm 4cm,clip=true]{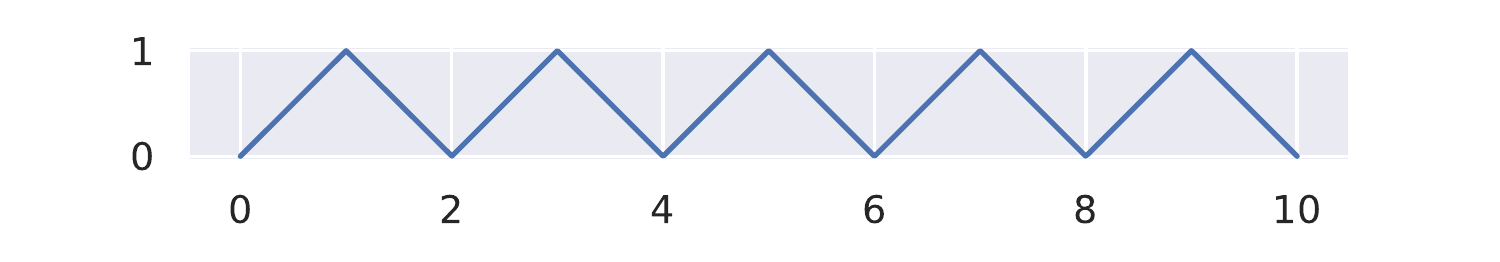}
      \caption{Plot of function $\alpha$.}
      \label{fig.alphacurve}
    \end{center}
  \end{figure}

  Let us define $\mathcal{I} \colon \mathbb{R}_{+}^{2} \to \mathbb{R}_{+}^{2}$ as follows
  \begin{equation*} 
    \left\{\begin{array}{ll}
      \left(1 - \delta(x,y) \right)(x, y) + \delta(x,y)(x, x), & \operatorname{for} \zeta(x,y) \in \left[0, \frac{1}{4}\right), \\
      \delta(x,y)(x, x) + \left(1-\delta(x,y)\right)(y, x), & \operatorname{for} \zeta(x,y) \in \left[\frac{1}{4}, \frac{3}{4}\right) \\
      \delta(x,y)(x, x) + \left(1-\delta(x,y)\right)(x, y), & \operatorname{for} \zeta(x,y) \in \left[\frac{3}{4}, 1\right),
    \end{array}\right.
  \end{equation*}
  where $\delta(x,y) := \alpha\left(\|(x, y)\|_{\infty}\right)$ and $\zeta(x,y) := \frac{1}{4}\|(x, y)\|_{\infty}-\left\lfloor\frac{1}{4}\|(x, y)\|_{\infty}\right\rfloor$.
  Then $\mathcal{I}(2,0) = (0,2)$ and $\mathcal{I}(4,0) = (4,0)$, so $(2,0) \leq(4,0)$.
  However, $\mathcal{I}(2,0) \not\leqslant_K \mathcal{I}(4,0)$.
  Furthermore, note that for $\left(x^{\prime}, y^{\prime}\right) \geq (x, y)$ we have
  \begin{equation*}
    \left\|\mathcal{I}\left(x^{\prime}, y^{\prime}\right)\right\|_{\infty} = \left\|\left(x^{\prime}, y^{\prime}\right)\right\|_{\infty} \geq \|(x, y)\|_{\infty} = \|\mathcal{I}(x, y)\|_{\infty}.
  \end{equation*}
  Note also that $\mathcal{I}\big(\beta(x, y)\big) = \beta \mathcal{I}(x, y)$, for all $\beta > 0$.

  Let $\sigma \colon \mathbb{R}_{+} \to \mathbb{R}_{+}$ be any activation function from the list above.
  If in the formula for $\mathcal{I}$ we put $\sigma(x)$ instead of $x$ and $\sigma(y)$ instead of $y$, we would get scalability of $\mathcal{I}$, i.e. $\mathcal{I}\big(\beta(x, y)\big) < \beta \mathcal{I}(x, y)$, for $\beta > 1$.

  \begin{figure}
    \begin{center}
      \includegraphics[width=0.9\columnwidth]{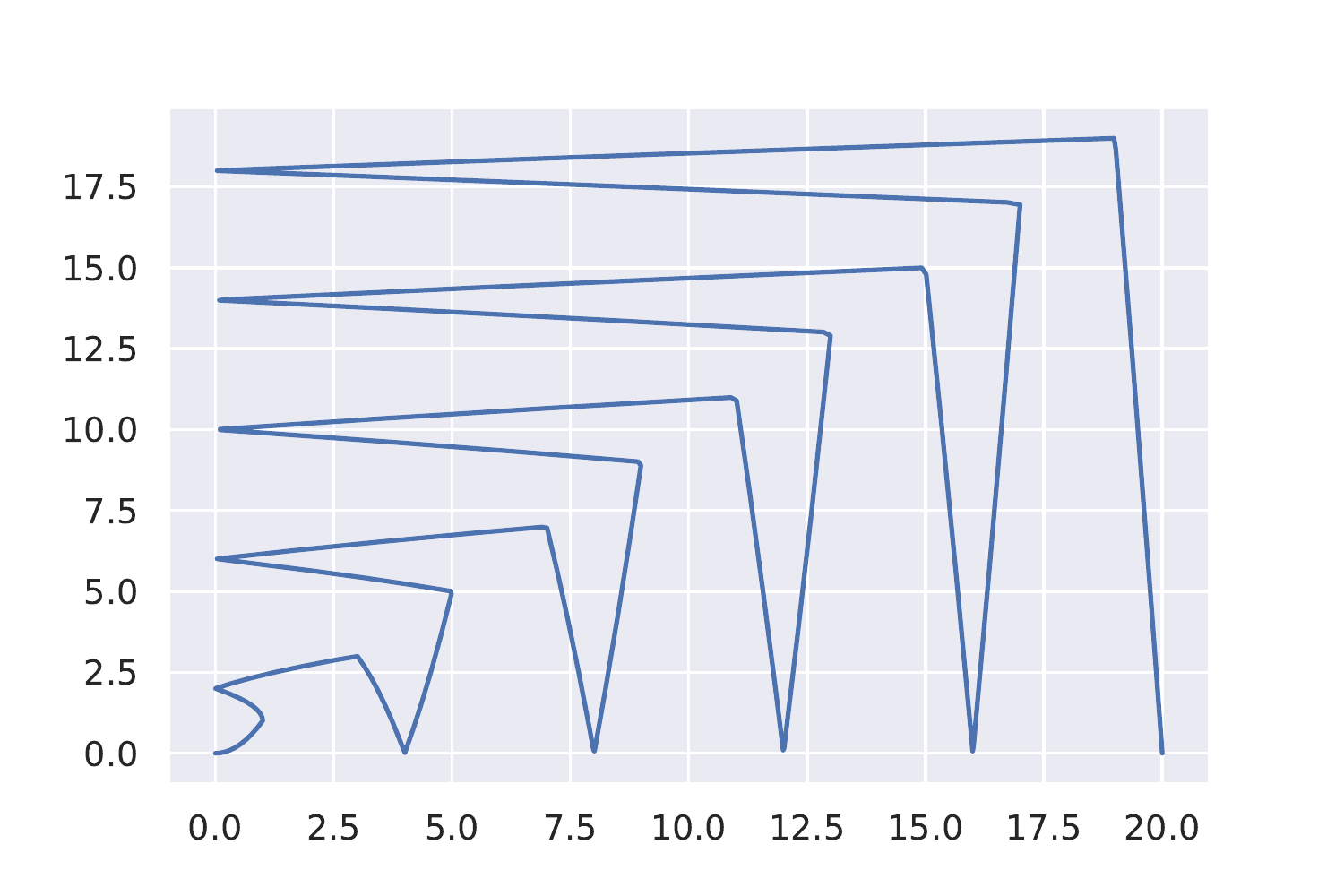}
      \caption{The image of interval $[0, 20]$ onto $\mathbb{R}^2$ via mapping $\mathcal{I}$.}
      \label{fig.drunkensailor}
    \end{center}
  \end{figure}


\end{example}

\begin{example}
  Consider activation function $\sigma \colon \mathbb{R}^2 \to \mathbb{R}^2$ given by the formula
  \begin{equation}
    \sigma(x_1, x_2) := \big( \sigma_1(x_1), \sigma_2(x_2) \big).
  \end{equation}
  Let
  \begin{equation}
    x :=
    \begin{bmatrix}
      x_1 \\
      x_2
    \end{bmatrix}, \quad
    W :=
    \begin{bmatrix}
      0 & -1 \\
      1 & 0
    \end{bmatrix}, \quad 
    b :=
    \begin{bmatrix}
      0 \\
      0
    \end{bmatrix}.
  \end{equation}
  Then $f(x) = \sigma(W x + b) = \big( \sigma_1(x_2), \sigma_2(x_1) \big)$.
  Let $K := \mathbb{R}^{N}_{+}$.
  If $\sigma_1 = \sigma_2$ is the sigmoid activation function and $x = (1, 2)$, $x^{\prime} = (2, 4)$, then $x \leqslant_K x^{\prime}$, but $f(x) \not\leqslant_K f(x^{\prime})$.
  However, $\|f(x)\|_v \leq \|f(x^{\prime})\|_v$ for $v = {\bf 1}$.
\end{example}

\section{Applications}

To illustrate the results obtained in the previous section a concrete application will be given.

\bibliographystyle{IEEEtran}
\bibliography{IEEEabrv,references}

\end{document}